\documentclass[runningheads]{llncs}
\usepackage{graphicx}
\usepackage{amsmath}
\usepackage{amssymb}
\usepackage{subfig}
\newcommand{\R}{\mathbb{R}}
\newcommand{\ul}{\mathbf}
\newcommand{\bp}{\mathbf{p}}
\newcommand{\bx}{\mathbf{x}}
\newcommand{\bR}{\mathbf{R}}
\newcommand{\bn}{\mathbf{n}}
\newcommand{\be}{\mathbf{e}}
\newcommand{\desda}{\Leftrightarrow}

\DeclareMathAlphabet\gothic{U}{euf}{m}{n}
\begin{document}
\title{Total Variation and Mean Curvature PDEs on $\R^{d} \rtimes S^{d-1}$}
\titlerunning{Total Roto-Translation Variation Flow on $\R^{d} \rtimes S^{d\!-\!1}$}
%
\author{Remco Duits$^1$, Etienne St-Onge$^2$, Jim Portegies$^1$ and Bart Smets$^1$}
\authorrunning{R.Duits, E.St.-Onge, J.W.Portegies, B.M.N.Smets}
%
\institute{$^1$CASA, Eindhoven University of Technology, the Netherlands, \\
$^2$ SCIL, Sherbrooke Connectivity Imaging Lab, Canada}
\maketitle              
\begin{abstract}
Total variation regularization and total variation flows
(TVF) have been widely applied for image enhancement and denoising.
To include a generic preservation of crossing curvilinear structures in TVF we lift images to the homogeneous space \mbox{$\mathbb{M}=\R^{d}\rtimes S^{d\!-\!1}$} of positions and orientations as a Lie group quotient in $SE(d)$. For $d=2$ this is called  `total roto-translation variation' by Chambolle \& Pock.
We extend this to $d=3$, by a PDE-approach with a limiting procedure for which we prove convergence. We also include a Mean Curvature Flow (MCF) in our PDE model on $\mathbb{M}$. This was first proposed for $d=2$ by Citti et al. and we extend this to $d=3$. Furthermore, for $d=2$ we take advantage of locally optimal differential frames in invertible orientation scores (OS).

We apply our TVF and MCF in the denoising/enhancement of crossing fiber bundles in DW-MRI. In comparison to data-driven diffusions, we see a better preservation of bundle boundaries and angular sharpness in fiber orientation densities at crossings. We support this by error comparisons on a noisy DW-MRI phantom.
We also apply our TVF and MCF in enhancement of crossing elongated structures in 2D images via OS,
and compare the results to nonlinear diffusions (CED-OS) via OS.

\end{abstract}
\section{Introduction} 

In the last decade, many PDE-based image analysis techniques for tracking and enhancement of curvilinear structures took advantage of lifting the image data to the homogeneous space $\mathbb{M}=\R^{d} \rtimes S^{d\!-\!1}$ of d-dimensional positions and orientations, cf.~\!\cite{DuitsPhD,Citti,Vogt,Boscain2018,bekkersPhD,Chambolle}. The precise definition of this homogeneous space follows in the next subsection. Set-wise it can be seen as a Cartesian product \mbox{$\mathbb{M}=\R^d\times S^{d\!-\!1}$}. Geometrically it can be equipped with a roto-translation equivariant geometry and topology beyond the usual isotropic Riemannian setting.

Typically, these PDE-based image analysis techniques
involve flows that implement morphological and (non)linear scale spaces or solve variational models.
The key advantage of extending the image domain from $\R^d$ to the higher dimensional lifted space $\mathbb{M}$ is that the PDE-flows do not suffer from crossings as fronts can pass without collision. This idea was shown for image enhancement in~\mbox{\cite{Fran2009,Sanguinetti}}. In \cite{Fran2009} the method of coherence enhancing diffusion (CED) \cite{Weic99b}, was lifted to $\mathbb{M}$ in a diffusion flow method called "coherence enhancing diffusion on invertible orientation scores" (CED-OS), that is recently generalized to 3D \cite{JanssenJMIV}.
Also for geodesic tracking methods, it helps that crossing line structures are disentangled in the lifted data. Geodesic flows prior to steepest descent also rely on (related \cite{Schmidt,BekkersSIAM,Chambolle}) PDEs on $\mathbb{M}$ commuting with roto-translations. They can account for crossings/bifurcations/corners~\cite{bekkersPhD,duitsmeestersmirebeauportegies,Chambolle}.

Nowadays PDE-flows on orientation lifts of \emph{3D images} are indeed relevant for applications such as fiber enhancement \cite{CreusenDuits,Vogt,Reisert,DuitsJMIV2013} and fiber tracking \cite{PortegiesPhD} in Diffusion-Weighted Magnetic Resonance Imaging (DW-MRI),
and in enhancement \cite{JanssenJMIV} and tracking \cite{cohen20183d} of blood vessels in 3D images.

As for PDE-based image denoising and enhancement, total variation flows (TVF) are more popular than nonlinear diffusion flows. Recently, Chambolle \& Pock generalized TVF  from $\R^2$ to $\mathbb{M}=\R^{2} \times S^{1}$, via `total roto-translation variation' (TV-RT) flows \cite{Chambolle} of 2D images. They employ (a)symmetric Finslerian geodesic models on $\mathbb{M}$ cf.~\!\cite{duitsmeestersmirebeauportegies}. As TVF falls short on invariance w.r.t. monotonic co-domain transforms, we also consider a Mean Curvature Flow (MCF) variant in our PDE model on $\mathbb{M}$, as proposed for 2D (i.e. $d=2$) by Citti et al.~\cite{Sanguinetti}.

To get a visual impression of how such PDE-based image processing on lifted images (orientation scores) works, for the case of tracking and enhancement of curvilinear structures in images
see Fig.~\!\ref{fig:intro}.
In the 3rd row of Fig.~\!\ref{fig:intro}, and henceforth, we visualize
a lifted image $U: \R^{3} \rtimes S^{2} \to \R^+$ by a grid of angular profiles
$\{\, \mu \, U(\mathbf{x},\mathbf{n})\, \mathbf{n}\;|\; \mathbf{x} \in \mathbb{Z}^{3}, \mathbf{n} \in S^{2}\,\}$, 
with fixed $\mu>0$.
\begin{figure}
\centerline{
\includegraphics[width=\hsize]{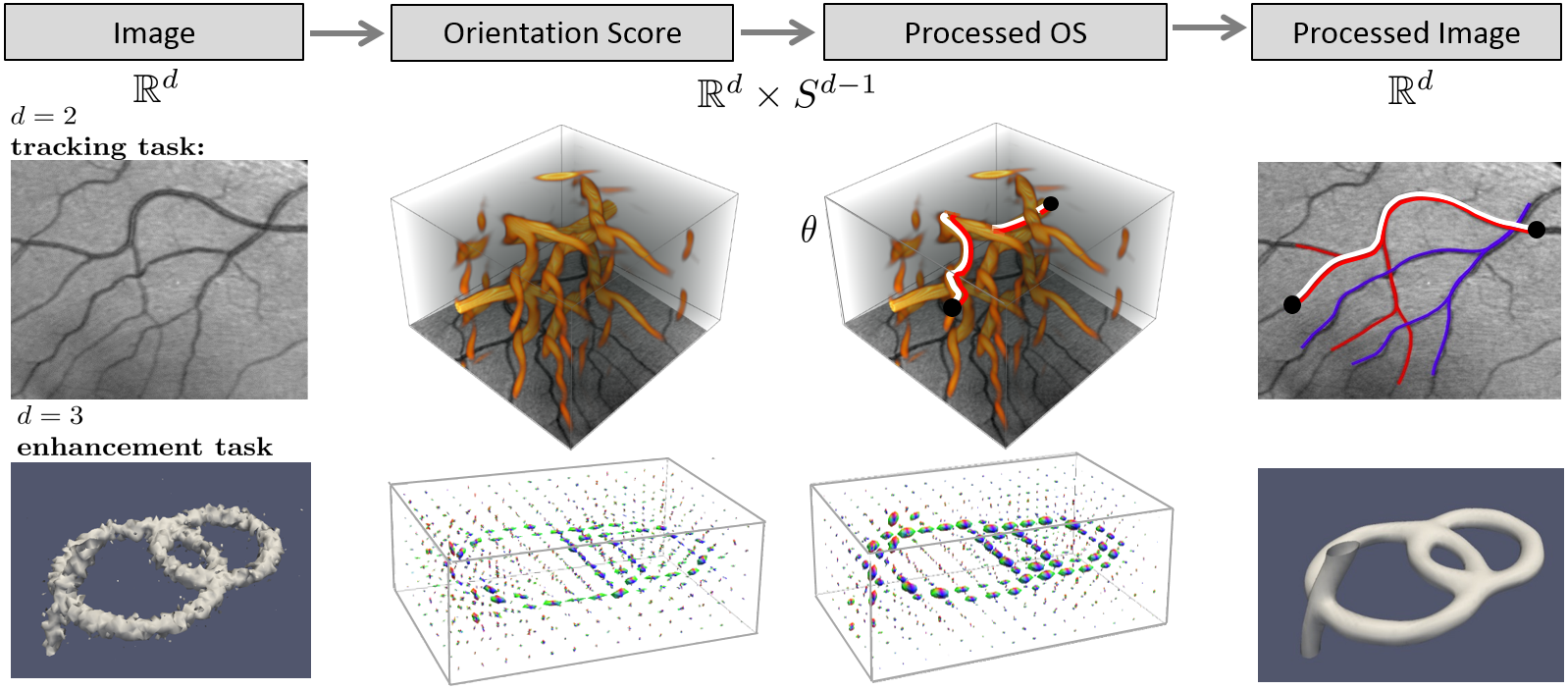}
}
\caption{Instead of direct PDE-based processing of an image, we apply PDE-based processing on a lifted image $U: \mathbb{R}^{d} \rtimes S^{d-1} \to \R$
(e.g. an orientation score: OS). The OS is obtained by convolving
the image with a set of rotated wavelets allowing for stable reconstruction \cite{DuitsPhD,bekkersPhD,JanssenJMIV}.
2nd row: Vessel-tracking in a 2D image via geodesic PDE-flows in OS that underly TVF: \cite{bekkersPhD,duitsmeestersmirebeauportegies,Chambolle}, with $\ul{n}=(\cos \theta, \sin \theta)^T \in S^1$.
3rd row: CED-OS diffusion of a 3D image \cite{JanssenJMIV,DuitsJanssen} visualized as a field of angular profiles.
In this article we study image enhancement and denoising via TVF and MCF on $\mathbb{M}=\R^{d} \rtimes S^{d\!-\!1}$ and compare to nonlinear diffusion methods on $\mathbb{M}$ (like CED-OS).
}
\label{fig:intro}
\end{figure}

The main contributions of this article are:
\vspace{-0.03cm}
\begin{itemize}
\item  We set up a geometric PDE flow framework on $\mathbb{M}$ including TVF, MCF, and diffusion. We tackle the PDEs by a basic limiting procedure. For TVF we prove convergence via Gradient flow theory by Brezis-Komura \cite{Brezis,Ambrosio}.
\item  We extend TVF on $\mathbb{M}$ \cite{Chambolle}, and MCF on $\mathbb{M}$ \cite{Sanguinetti} to the 2D and 3D setting.
\item  We apply TVF and MCF in the denoising and enhancement of crossing fiber bundles in fiber orientation density functions (FODF) of DW-MRI data. In comparison to data-driven diffusions, we show a better preservation of bundle boundaries and angular sharpness with TVF and MCF. 
We support this observation by error comparisons on a noisy DW-MRI phantom.
\item We include locally optimal differential frames (LAD) \cite{DuitsJanssen} in invertible orientation scores (OS), and propose crossing-preserving denoising methods TVF-OS, MCF-OS. We show benefits of LAD inclusion on 2D data.
\item  We compare TVF-OS, MCF-OS to CED-OS on 2D images.
\end{itemize}

\section{Theory}

\subsection{The Homogeneous Space $\mathbb{M}$ of Positions and Orientations}

Set $d\in \{2,3\}$.
Consider the rigid body motion group, $SE(d)=\R^{d} \rtimes SO(d)$.
It acts transitively on $\R^{d} \times S^{d\!-\!1}$ by
\[
g \odot (\ul{x}',\ul{n}') =
(\ul{R}\ul{x}'+\ul{x},\ul{R} \ul{n}'), \textrm{ for all }g= (\ul{x},\ul{R}) \in SE(d),\  (\ul{x}',\ul{n}') \in \R^{d} \times S^{d\!-\!1}.
\]
Now set $\ul{a} \in S^{d-1}$ as an a priori reference axis, say $\ul{a}=(1,0)$ if $d=2$ and $\ul{a}=\be_z=(0,0,1)^T$ if $d=3$.
The homogeneous space
of positions and orientations is the partition of left-cosets:
\[
\mathbb{M}:=\R^{d} \rtimes S^{d\!-\!1}:=SE(d)/H,
\]
in $SE(d)$ and
$H=\{g \in SE(d) \;|\; g \odot (\ul{0},\ul{a})=(\ul{0},\ul{a})\}$.
For $d=2$ we have $\mathbb{M}\equiv SE(2)$. For $d=3$ we have that
\begin{equation} \label{subgroup}
H=\{h_{\alpha}:=(\mathbf{0}, \mathbf{R}_{\mathbf{a},\alpha}) \;|\; \alpha \in [0,2\pi)\},
\end{equation}
where $\bR_{\mathbf{a},\alpha}$ denotes a (counter-clockwise) rotation around $\mathbf{a} =\be_z$.
Recall that by the definition of the left-cosets one has
$
H = \{\mathbf{0}\} \times SO(2), \textrm{ and }g_{1}\sim g_{2} \desda g_{1}^{-1}g_2 \in H$. This means that
for $
g_{1}=(\mathbf{x}_{1},\mathbf{R}_{1})$, $g_{2}=(\mathbf{x}_{2},\mathbf{R}_{2})$ one has
\[
g_{1}\sim g_{2} \desda \mathbf{x}_{1}=\mathbf{x}_{2} \textrm{ and }\exists_{\alpha \in [0,2\pi)}\;:\; \mathbf{R}_{1}= \mathbf{R}_{2} \mathbf{R}_{\mathbf{a},\alpha}.
\]
The equivalence classes $[g]=\{g' \in SE(3)\; |\; g'\sim g\}$ are usually just denoted by $\mathbf{p}=(\mathbf{x},\mathbf{n})$ as they consist of all rigid body motions
$g=(\mathbf{x},\mathbf{R}_{\mathbf{n}})$ that map reference point $(\mathbf{0},\mathbf{a})$ onto $(\mathbf{x},\mathbf{n}) \in \R^3 \rtimes S^2$\ :
\begin{equation}\label{eq:qoutfromgroup}
g \odot (\mathbf{0},\mathbf{a})=(\mathbf{x},\mathbf{n}).
\end{equation}
On tangent bundle
$T(\mathbb{M})=\{(\bp,\dot{\bp})\;|\; \bp \in \mathbb{M}, \dot{\bp} \in T_{\bp}(\mathbb{M})\}$ we set metric tensor:
\begin{equation} \label{metric}
\begin{array}{l}
\left.\mathcal{G}_{\gothic{e}}\right|_{\mathbf{p}}(\dot{\mathbf{p}},\dot{\mathbf{p}})=  D_S^{-1}|\dot{\bx} \cdot \bn|^2 + D_{A}^{-1}
\|\dot{\bn}\|^2 + \gothic{e}^{-2} D_{S}^{-1} \|\dot{\bx} \wedge \bn \|^2, \\[5pt]
\textrm{for all }\bp=(\mathbf{x},\mathbf{n}) \in \mathbb{M}, \dot{\bp}=(\dot{\bx},\dot{\bn}) \in T_{\bp}(\mathbb{M})
\end{array}
\end{equation}
with $0< \gothic{e} \ll 1$ fixed, and with constant $D_s>0$ costs for spatial motions and
constant $D_A>0$ costs for angular motions.
For the sub-Riemannian setting ($\gothic{e}=0$) we set $\left.\mathcal{G}_{0}\right|_{\mathbf{p}}(\dot{\mathbf{p}},\dot{\mathbf{p}})=D_S^{-1}|\dot{\bx} \cdot \bn|^2 + D_{A}^{-1}
\|\dot{\bn}\|^2$ and constrain $\mathcal{G}_{0}$ to the sub-tangent bundle given by
$\{((\mathbf{x},\mathbf{n}),(\dot{\mathbf{x}},\dot{\mathbf{n}}))\;|\; \dot{\ul{x}} \wedge \ul{n}=\ul{0} \}$.

Then (\ref{metric}) sets the Riemannian (resp. sub-Riemannian) gradient:
\begin{equation} \label{grad}
\begin{array}{l}
\nabla_{\gothic{e}}\, U(\bp) = (\mathcal{G}^{-1}_{\gothic{e}} {\rm d} U)(\bp) \equiv
\\[3pt]
\left(D_{S}\, \ul{n} (\ul{n} \cdot \nabla_{\R^d} U(\bp))+\gothic{e}^2\; D_{S} (I -\ul{n} \otimes \ul{n})\nabla_{\R^d} U(\bp)\,,\, D_A \,\nabla_{S^{d\!-\!1}}U(\bp)\right)^T\!, \\[7pt]
\nabla_0 \, U(\bp) 
=
\left(\;D_{S}\, \ul{n} (\ul{n} \cdot \nabla_{\R^d} U(\bp))\;,\; D_A \,\nabla_{S^{d\!-\!1}}U(\bp)\;\right)^T,
\end{array}
\end{equation}
for all differential functions $U \in C^{1}(\mathbb{M},\mathbb{R})$.

We have particular interest for $U \in C^{1}(\mathbb{M},\mathbb{R})$ that are `orientation lifts' of input image $f:\Omega_f \to \R^{+}$. Such $U$ are compactly supported within \begin{equation} \label{omega}
\Omega = \Omega_f \times S^{d-1} \subset \mathbb{M}.
\end{equation}
Such a lift may be (the real part of) an invertible orientation score \cite{DuitsAMS1} (cf.~Fig.~\ref{fig:intro}), a channel-representation \cite{Felsberg2}, a lift by Gabor wavelets \cite{Baspinar2018}, or a fiber orientation density \cite{portegies_improving_2015}, where in general the absolute value
$|U(\mathbf{x},\mathbf{n})|$ is a probability density of finding a fiber structure at position $\mathbf{x} \in \R^d$ with local orientation $\mathbf{n} \in S^{d\!-\!1}$.

The corresponding norm of the gradient equals
\begin{equation} \label{norm}
\|\nabla_{\gothic{e}} U(\bp)\|_{\gothic{e}}=
\sqrt{\left.\mathcal{G}_{\gothic{e}}\right|_{\bp}(\nabla_{\gothic{e}} U(\bp),\nabla_{\gothic{e}} U(\bp))}.
\end{equation}
We set the following volume form on $\mathbb{M}$:
\begin{equation}
\begin{array}{ll}
\label{volumeform}
{\rm d}\mu = D_{S}^{-1} {\rm d}\mathbf{x} \; \wedge
\;
D_{A}^{-1} {\rm d}\sigma_{S^{d-1}}.
\end{array}
\end{equation}
This induces the following (sub-)Riemannian divergence
\begin{equation} \label{div}
\textrm{div}\, \mathbf{v}=
\left\{
\begin{array}{ll}
\textrm{div}_{\R^3} \mathbf{v} + \textrm{div}_{S^2}\mathbf{v} &\textrm{ for } \gothic{e}>0, \\
\ul{n} \cdot \nabla_{\R^3}(\mathbf{n} \cdot \mathbf{v}) + \textrm{div}_{S^2}\mathbf{v} &\textrm{ for } \gothic{e}=0.
\end{array}
\right.
\end{equation}
as the Lie derivative of ${\rm d}\mu$ along $\ul{v}$ is
$\mathcal{L}_{\mathbf{v}} {\rm d}\mu= (\textrm{div}\, \mathbf{v})\, \mu$.
TV on $\R^n$ is mainly built on the identity
$\nabla \cdot (f \mathbf{v})= f \; \nabla \cdot \mathbf{v} + \nabla f \cdot \mathbf{v}$. Similarly on
$\mathbb{M}$ one has:
\begin{equation} \label{divfv}
\textrm{div}( U \mathbf{v})(\mathbf{p})= U(\mathbf{p}) \, \textrm{div}\, \mathbf{v}(\mathbf{p}) + \left.\mathcal{G}_{\gothic{e}} \right|_{\mathbf{p}}\left(\nabla_{\gothic{e}} U(\bp), \mathbf{v}(\bp) \right),
\end{equation}
for all $\mathbf{p} \in \mathbb{M}$,
from which we deduce the following integration by parts formula:
\begin{equation} \label{intparts}
\int \limits_{\Omega} U(\mathbf{p}) \, \textrm{div}\, \mathbf{v}(\mathbf{p})\; {\rm d}\mu(\mathbf{p})= \int \limits_{\Omega} \left.
\mathcal{G}_{\gothic{e}} \right|_{\mathbf{p}}\left(\nabla_{\gothic{e}} U(\bp), \mathbf{v}(\bp) \right)\; {\rm d}\mu(\mathbf{p}),
\end{equation}
for all $U \in C^{1}(\Omega)$ and all smooth vector fields $\mathbf{v}$ vanishing at the boundary $\partial \Omega$. This formula allows us to build a weak formulation of TVF on $\mathbb{M}$.
\begin{definition}(weak-formulation of TVF on $\mathbb{M}$) \\
Let $U \in BV(\Omega)$ a function of bounded variation. Let $\chi_0(\Omega)$ denote the vector space of smooth vector fields
that vanish at the boundary $\partial \Omega$.
Then we define
\begin{equation} \label{start}
\begin{array}{ll}
TV_{\varepsilon}(U) &:= \sup \limits_{
{\scriptsize \begin{array}{c}
\psi \in C^{\infty}_c(\Omega) \\
\mathbf{v}
 \in \chi_0(\Omega) \\
 \|\mathbf{v}(\mathbf{p})\|^2_{\gothic{e}} + |\psi(\mathbf{p})|^2\leq 1
 \end{array}
}} \int \limits_{\Omega}\;\,
{\small
\left(
\begin{array}{c}
\varepsilon \\
U(\mathbf{p})
\end{array}
\right)
\cdot
\left(
\begin{array}{c}
\psi(\mathbf{p}) \\
{\rm \textrm{div}}\, \mathbf{v}(\mathbf{p})
\end{array}
\right)}\;
{\rm d}\mu(\mathbf{p})
\end{array}
\end{equation}
\end{definition}
For all $U \in BV(\Omega)$ we have $TV_0(U) \leq TV_{\varepsilon}(U) \leq
TV_{0}(U) + \varepsilon |\Omega|$.
\begin{lemma} \label{lemma:1}
Let $\varepsilon, \gothic{e} \geq 0$.
For $U \in C^{1}(\Omega,\mathbb{R})$  we have
\begin{equation} \label{result}
TV_{\varepsilon}(U)= \int_{\Omega} \sqrt{\|\nabla_{\gothic{e}}U(\mathbf{p})\|^2_{\gothic{e}} + \varepsilon^2}\;\, {\rm d}\mu(\mathbf{p}).
\end{equation}
For $U \in C^{2}(\mathbb{M},\R)$ and $\gothic{e}, \varepsilon>0$ we have $\partial \, TV_{\varepsilon}(U)= \mathrm{div} \circ
\left( \frac{\nabla_{\gothic{e}}(U)}{\sqrt{\|\nabla_{\gothic{e}}(U)\|^2_{\gothic{e}}+\varepsilon^2}}
\right)$.
\end{lemma}
\begin{proof}
First we substitute (\ref{intparts}) into (\ref{start}), then we apply Gauss theorem and use $\left. U\mathbf{v} \right|_{\partial \Omega}=0$. Then we apply Cauchy-Schwarz  on $V_{\mathbf{p}}:=\R \times T_{\mathbf{p}}(\mathbb{M}) $ for each $\mathbf{p} \in \mathbb{M}$, with inner product $(\epsilon_1, \ul{v}_1) \cdot (\epsilon_2, \ul{v}_2)= \epsilon_1 \epsilon_2 + \mathcal{G}_{\ul{p}}(\ul{v}_{1},\ul{v}_{2})$, which holds with equality iff the vectors are linearly dependent. Therefore we smoothly approximate $\frac{1}{\sqrt{\varepsilon^2+ \|\nabla_{\gothic{e}}U\|^2_{\gothic{e}}}}(\varepsilon , \nabla_{\gothic{e}}U)$ by $(\psi,\mathbf{v})$ to get
(\ref{result}). For $U \in C^{2}(\Omega,\R)$, $\delta \in C_{c}^{\infty}(\Omega,\R)$
we 
get
$
(\partial \, TV_{\varepsilon}(U), \delta)_{\mathbb{L}_{2}(\Omega)}=\lim \limits_{h \downarrow 0} \frac{TV_{\varepsilon}(U+h\, \delta)-TV_{\varepsilon}(U)}{h} \overset{(10)}{=} (\textrm{div} \circ
\left( \frac{\nabla_{\gothic{e}}(U)}{\sqrt{\|\nabla_{\gothic{e}}(U)\|^2_{\gothic{e}}+\varepsilon^2}}
\right),\delta)_{\mathbb{L}_{2}(\Omega)}$.
\end{proof}
\subsection{
Total-Roto Translation Variation, Mean Curvature Flows on $\mathbb{M}$}

Henceforth, we fix $\gothic{e}\geq 0$ and write $\nabla=\nabla_{\gothic{e}}$. We propose the following roto-translation equivariant enhancement PDEs on $\Omega\subset \mathbb{M}$, recall (\ref{omega}):
\begin{equation} \label{PDE}
\left\{
\begin{array}{rcl}
\frac{\partial W^{\varepsilon}}{\partial t}(\bp,t) &=&
\left(\|{\rm \nabla} W^{\varepsilon}(\bp,t)\|^2 + \varepsilon^2\right)^\frac{a}{2}
\left(\textrm{div}\, \circ \frac{\nabla W^{\varepsilon}(\cdot,t)}{
\left(
\|\nabla W^{\varepsilon}(\cdot,t)\|^2_{\gothic{e}} +
\varepsilon^2
\right)^{\frac{b}{2}}}
\right)(\bp),  \\[12pt] 
 0&=&\mathbf{N}(\mathbf{x}) \cdot \nabla_{\mathbb{R}^d} W^{\varepsilon}(\bx,\bn,0) \qquad \ \ \bp=(\mathbf{x},\mathbf{n}) \in \partial \Omega, \\[6pt]
W^{\varepsilon}(\bp,0)&=& U(\bp) \qquad \qquad \qquad \qquad \qquad\, \bp=(\mathbf{x},\mathbf{n}) \in \Omega,
\end{array}
\right. \,
\end{equation}
with evolution time $t\geq 0$, $0<\varepsilon \ll 1$, and
with parameters $a,b \in \{0,1\}$. Regarding the boundary of $\Omega$ we note that $\bp=(\mathbf{x},\mathbf{n}) \in \partial \Omega \Leftrightarrow \mathbf{x} \in \partial \Omega_f, \mathbf{n} \in S^2$. We use Neumann boundary conditions as $\ul{N}(\ul{x})$ denotes the normal at $\mathbf{x} \in \partial \Omega_f$.

For $\{a,b\}=\{1,1\}$ we have a
geometric Mean Curvature Flow (MCF) PDE.
For $\{a,b\}=\{0,1\}$ we have a Total Variation Flow (TVF) \cite{Chambolle}.
For $\{a,b\}=
\{0,0\}$ we obtain a linear diffusion for which exact smooth solutions exist \cite{Portegies}.
\begin{remark}
By the product rule (\ref{divfv}) the right-hand side  of (\ref{PDE}) for $\varepsilon \downarrow 0$ becomes
\begin{equation} \label{PDE2}
\frac{\partial W^{0}}{\partial t}= \|\nabla W^0\|^{a-b}\Delta W^0 + 2b \, \overline{\kappa}_{I}\; \|\nabla W^0\|^{a},
\end{equation}
with the mean curvature $\overline{\kappa}_{I}(\bp,t)$ of level set
$\{\mathbf{q} \in \mathbb{M}\;|\; W^{0}(\mathbf{q},t)=W^{0}(\bp,t)\}$, akin to \cite[ch;3.2]{sapiro_geometric_2006},
and with (possibly hypo-elliptic) Laplacian $\Delta=\textrm{div} \circ \nabla $.
\end{remark}
\begin{remark}
For MCF and TVF smooth solutions to the PDE (\ref{PDE}) exist only under special circumstances. This lack of regularity is an advantage in image processing
to preserve step-edges and plateaus in images, yet it forces us to define a concept of weak solutions. Here, we distinguish between MCF and TVF:

For MCF one relies on viscosity solution theory developed by Evans-Spruck \cite{Evans-Motion-1991}, see also \cite{Giga-Generalized-1991,Sato-Interface-1994} for the case of MCF with Neumann boundary conditions. 
 In \cite[Thm 3.6]{Sanguinetti} existence of $C^{1}$-viscosity solutions is shown for $d=2$. 

For TVF we will rely on gradient flow theory by Brezis-Komura~\cite{Brezis,Ambrosio}.
\end{remark}
\begin{remark}
Convergence of the solutions w.r.t. $\gothic{e} \downarrow 0$ is clear from the exact solutions for $\{a,b\}=\{0,0\}$, see \cite[ch:2.7]{Portegies}, and is also  addressed for  MCF \cite{Sanguinetti,Baspinar-Minimal-2018}.
For TVF one can rely on \cite{duitsmeestersmirebeauportegies}.
Next we focus on convergence results for $\varepsilon \downarrow 0$.
\end{remark}

\subsection{Gradient-Flow formulations and Convergence}

The total variation flow can be seen as a gradient flow of a lower-semicontinuous, convex functional in a Hilbert space, as we explain next.

If $F: H \to [0,\infty]$ is a proper (i.e. not everywhere equal to infinity), lower semicontinuous, convex functional on a Hilbert space $H$, the subdifferential of $F$ in a point $u$ in the finiteness domain of $F$ is defined as
\[
\partial F (u) := \left\{ z \in H | \ (z, v-u) \leq F(v) - F(u) \text{ for all } v \in H \right\}.
\]
The subdifferential is closed and convex, and thereby it has an element of minimal norm, called ``the gradient of $F$ in $u$'' denoted by $\mathrm{grad} F(u)$.
Let $u_0$ be in the closure of the finiteness domain of $F$.
By Brezis-Komura theory, \cite{Brezis}, \cite[Thm 2.4.15]{Ambrosio}
there is a unique locally absolutely continuous curve $u: [0,\infty) \to H$ s.t.
\[
- u'(t) = \mathrm{grad} F(u(t)) \text{ for a.e. } t > 0 \text{ and } \lim_{t \downarrow 0} u(t) = u_0.
\]
We call $u:[0,\infty) \to H$ the gradient flow of $F$ starting at $u_0$.

The function $TV_\epsilon: L^2(\Omega) \to [0,\infty]$ is lower-semicontinuous and convex for every $\epsilon \geq 0$.  This allows us to generalize solutions to the PDE (\ref{PDE}) as follows:
\begin{definition}
Let $U \in \Xi:= BV(\Omega) \cap \mathbb{L}_{2}(\Omega)$. We define by $t \mapsto W^\epsilon(\cdot, t)$ the gradient flow of $TV_\epsilon$ starting at $U$.
\end{definition}
\begin{remark}
A smooth solution $W^\epsilon$ to (\ref{PDE}) with $\{a,b\}=\{0,1\}$ is a gradient flow. 
\end{remark}
\begin{theorem} \label{th:conv} (strong $\mathbb{L}_{2}$-convergence, stability and accuracy of TV-flows) \\
Let $U \in \mathbb{L}_{2}(\Omega)$ and let $W^{\varepsilon}$ be the
gradient flow of $TV_{\varepsilon}$ starting at $U$ and $\varepsilon,\gothic{e} \geq 0$. Let $t\geq 0$. Then
\[
\lim \limits_{\varepsilon \downarrow 0} W^{\varepsilon}(\cdot,t) = W^{0}(\cdot,t)\textrm{ in }\mathbb{L}_{2}(\Omega).
\]
More precisely, for $U \in BV(\Omega)$, we have for all $t \geq 0$:
\[
\|W^\varepsilon(\cdot, t) - W^0(\cdot, t)\|_{\mathbb{L}_{2}(\Omega)} \leq 8 \Big(\|U\|_{L^2(\Omega)} (TV_0(U)+\delta) \delta t^2\Big)^{1/5} \textrm{ with }\delta=\varepsilon |\Omega|
\]
\end{theorem}
Theorem \ref{th:conv} follows from the following general result, if we take $F = TV_0$, $G=TV_\epsilon$, $\delta = \epsilon |\Omega|$.
\begin{theorem}
\label{th:main-abstract}
Let $F:H \to [0,\infty]$ and $G:H \to [0,\infty]$ be two proper, (i.e. not everywhere equal to infinity),
lower semi-continuous, convex functionals on a Hilbert space $H$, such that
\[
F(u) - \delta \leq G(u) \leq F(u) + \delta
\]
for all $u \in H$. Let $u_0, v_0 \in H$ be such that $F(u_0) \leq E$ and $G(v_0) \leq E$ and $\|u_0\| \leq M$, $\|v_0\| \leq M$. The gradient flow $u:[0,\infty) \to H$ of $F$ starting at $u_0$, and the gradient flow $v:[0,\infty) \to H$ of $G$ starting at $v_0$ satisfy
\[
\| u(t) - v(t) \|_{H} \leq 16 (M E \delta t^2)^{1/5} + \| u_0 - v_0 \|_H
\]
for all $0 \leq t \leq E^6 M^6 / \delta^9$.
\end{theorem}
For a proof see appendix A.

\subsection{Numerics}

We implemented the PDE system (\ref{PDE}) by Euler forward time discretization, relying on standard B-spline or linear interpolation techniques for derivatives in the underlying tools of the gradient on $\mathbb{M}$ given by (\ref{grad}) and the divergence on $\mathbb{M}$ given by (\ref{div}). For details see~\!\cite{Fran2009,CreusenDuits}. Also, the explicit upperbounds for stable choices of stepsizes can be derived by the Gershgorin circle theorem, \cite{Fran2009,CreusenDuits}.

The PDE system (\ref{PDE}) can be re-expressed by a left-invariant PDE on $SE(d)$ as done in related previous works by several researchers \cite{DuitsAMS1,Fran2009,Sanguinetti,CreusenDuits,Boscain2018,Citti}. For $d=2$ this is straightforward as $SE(2) \equiv \R^{2} \rtimes S^{1}$. For $d=3$ and $\gothic{e}=0$ one has
\begin{equation} \label{corr}
\begin{array}{l}
\textrm{div}\, \mathbf{v} \leftrightarrow \mathcal{A}_{3}\tilde{v}^3 + \mathcal{A}_{3}\tilde{v}^4 + \mathcal{A}_{5}\tilde{v}^5, \
\nabla_0 W \leftrightarrow (D_{S} \mathcal{A}_{3}\tilde{W}, D_{A} \mathcal{A}_{4}\tilde{W}, D_A \mathcal{A}_{5}\tilde{W} )^T
\end{array}
\end{equation}
where $\{\mathcal{A}_i\}$ is a basis of vector fields on $SE(3)$ given by
$
(\mathcal{A}_{i}f)(g)=\lim \limits_{t \downarrow 0}\frac{f(g e^{t A_i})- f(g)}{t}
$
with a Lie algebra basis $\{A_i\}$ for $T_{e}(SE(3))$ as in \cite{DuitsPhD,CreusenDuits}, and with $\tilde{W}(\mathbf{x},\mathbf{R},t)=W(\mathbf{x},\mathbf{R} \mathbf{a},t)$, $\tilde{v}^i(\mathbf{x},\mathbf{R},t)=v^i(\mathbf{x},\mathbf{R} \mathbf{a},t)$. We used (\ref{corr}) to apply  discretization on $SE(3)$ \cite{CreusenDuits} in the software developed by Martin et al.\cite{LieAnalysis}, to our PDEs of interest (\ref{PDE}) on $\mathbb{M}$ for $d=3$.

\begin{remark}
The Euler-forward discretizations are not unconditionally stable. For $a=b=0$,
the Gerhsgorin circle theorem \cite[ch.4.2]{CreusenDuits} gives the stability bound
\[
\Delta t \leq (\Delta t)_{crit}:={\small \left(\frac{(d-1)D_A+D_S}{2h^2}+\frac{(d-1)D_A}{2h_a^2}\right)^{-1}},
\]
when using linear interpolation with spatial stepsize $h$ and angular stepsize $h_{a}$.
In our experiments, for $d=2$ we set $h=1$ and for $d=3$ we took $h_{a}=\frac{\pi}{25}$ using an almost uniform spherical sampling from a tessellated icosahedron with $N_A=162$ points.
TVF required smaller times steps when $\varepsilon$ decreases. Keeping in mind (\ref{PDE2}) but then applying the product rule (\ref{divfv}) to the case $0<\varepsilon \ll 1$, we concentrate on the first term as it is of order $\varepsilon^{-1}$ when the gradient vanishes. Then we find $\Delta t \leq \varepsilon \cdot (\Delta t)_{crit}$ for TVF. For MCF we do not have this limitation.
\end{remark}

\section{Experiments}

In our experiments, we aim to enhance contour and fiber trajectories in medical images and to remove noise.
Lifting the image $f:\R^{d} \to \R$ towards its orientation lift $U:\mathbb{M} \to \R$ defined on the
space of positions and orientations $\mathbb{M}=\R^{d} \rtimes S^{d-1}$ preserves crossings
\cite{Fran2009} and avoids leakage of wavefronts \cite{duitsmeestersmirebeauportegies}.

For our experiments for $d=3$ the initial condition $U: \mathbb{M} \to \R^+$ is a fiber orientation density function (FODF) obtained from DW-MRI data \cite{portegies_improving_2015}.

For our experiments for $d=2$ the initial condition $U$ is
an invertible orientation score (OS) and we
adopt the flows in (\ref{PDE}) via
locally adaptive frames~\cite{DuitsJanssen}.
For both $d=3$ (Subsection~\ref{ch:d3}) and $d=2$ (Subsection~\ref{ch:d2}),
we show advantages of TVF and MCF over crossing-preserving diffusion flows \cite{Fran2009,CreusenDuits} on $\mathbb{M}$. We set $\gothic{e}=0$ in all presented experiments as it gave better results than $\gothic{e}>0$.

\subsection{TVF \& MCF on $\R^3 \rtimes S^2$ for Denoising FODFs in DW-MRI \label{ch:d3}}

In DW-MRI image processing one obtains a field of angular diffusivity profiles (orientation density function) of water-molecules. A high diffusivity in particular orientation correlates to biological fibers structure, in brain white matter, along that same direction. Crossing-preserving enhancement of FODF fields $U:\mathbb{M} \to \R^+$ helps to better identify structural pathways in brain white matter, which is relevant for surgery planning, see for example \cite{Meesters,portegies_improving_2015}.

For a quantitative comparison we applied TVF, MCF and PM diffusion \cite{CreusenDuits} to denoise a popular synthetic FODF $U: \mathbb{M} \to \R^+$ from the ISBI-HARDI 2013 challenge \cite{ISBI}, with realistic noise profiles. In Fig.~\ref{fig:synthetic}, we can observe the many crossing fibers in the dataset. Furthermore, we depicted the absolute $\mathbb{L}_{2}$-error
$t \mapsto \|U - \Phi_t(U)(\cdot)\|_{\mathbb{L}_{2}(\mathbb{M})}$
as a function of the evolution parameter $t$, where $\Phi_t(U)=W_{\varepsilon}(\cdot,t)$
with optimized $\varepsilon=0.02$ in the case of TVF (in green), and MCF (in blue), and where $\Phi_t$
is the PM diffusion evolution \cite{CreusenDuits} on $\mathbb{M}$ with optimized PM parameter $K=0.2$ (in red). 
We also depict results for $K=0.1, 0.4$ (with the dashed lines) and $\varepsilon=0.01, 0.04$.
We see that the other parameter settings provide on average worse results, justifying our 
optimized parameter settings. We set $D_{S}=1.0$, $D_A=0.001$, $\Delta t=0.01$.
We observe that:
\begin{itemize}
\item TVF can reach lower error values than MC-flow with adequate $\Delta t=0.01$,
\item MCF provides more stable errors for all $t>0$, than TV-flow w.r.t. $\epsilon>0$,
\item TVF and MCF produce lower error values than PM-diffusion,
\item PM-diffusion provides the most variable results for all $t>0$.
\end{itemize}
\begin{figure}[!h]
\centering
  \vspace{-0.2cm}
  \includegraphics[width=0.75\textwidth]{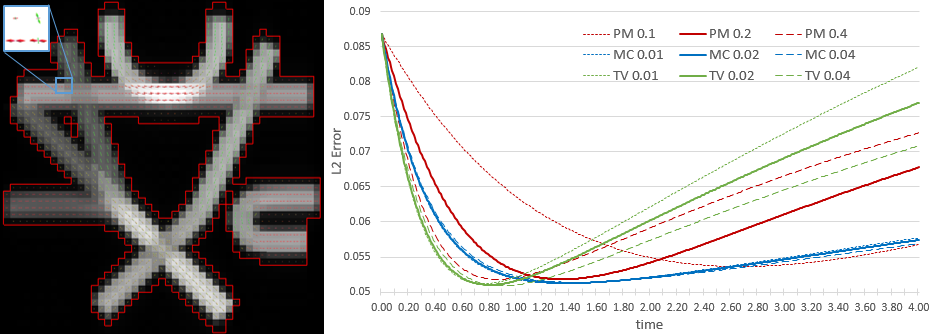}
    \vspace{-0.25cm}
  \caption{
\label{fig:synthetic}
Quantitative comparison of denoising a fiber orientation density function (FODF)
obtained by
(CSD) \cite{Tournier} from a benchmark DW-MRI dataset~\cite{ISBI}.}
\end{figure}
\begin{figure}[!h]
\centering
  \includegraphics[width=0.84\textwidth]{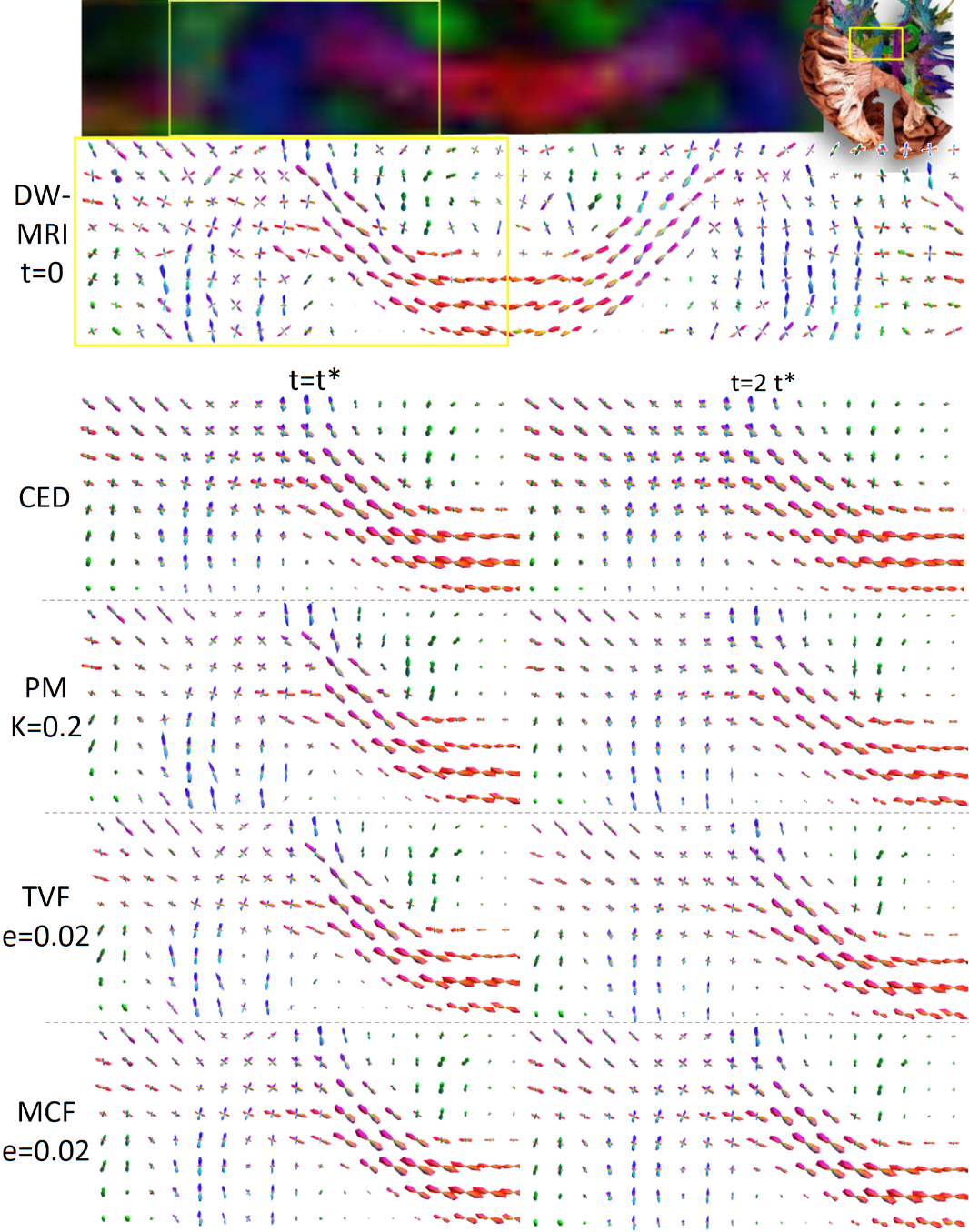}
\caption{
\label{fig:2}
Qualitative comparison of denoising a FODF obtained by
(CSD) \cite{Tournier} from a standard DW-MRI dataset (with $b=1000 s/mm^2$ and $54$ gradient directions).
For the CSD we used up to 8th order spherical harmonics, and the FODF is then spherically sampled on a tessellation of the icosahedron with 162 orientations.}
\end{figure}

For a qualitative comparison we applied TVF, MCF, PM diffusion and linear diffusion to a FODF field $U: \mathbb{M} \to \R^+$
obtained from a standard DW-MRI dataset  (with $b=1000 s/mm^2$, $54$ gradient directions) via constrained spherical deconvolution (CSD) \cite{Tournier}.
%
%
See Fig.~\!~\ref{fig:2}, where for each method, we used the optimal parameter settings with the artificial data-set. We see that
\begin{itemize}
\item all methods perform well on the real datasets. Contextual alignment of the angular profiles better reflects the anatomical fiber bundles,
\item MCF and TVF better preserve boundaries and angular sharpness,
\item MCF better preserves the amplitude at crossings at longer times.
\end{itemize}

\subsection{TVF \& MCF on $\R^2 \rtimes S^1$
for 2D Image Enhancement/Denoising \label{ch:d2}}

The initial condition for our TVF/MCF-PDE (\ref{PDE}) is set by an orientation score \cite{DuitsPhD} of image $f:\R^2 \to \R$ given by $\mathcal{W}_{\psi}f(\mathbf{x},\mathbf{n})=(\psi_{\mathbf{n}} \star f)(\mathbf{x})$ where $\star$ denotes correlation and $\psi_{\mathbf{n}}$ is the rotated wavelet aligned with $\mathbf{n} \in S^{1}$.
For $\psi$ we use a cake-wavelet \cite[ch:4.6]{DuitsPhD} $\psi$ with standard settings \cite{LieAnalysis}.
Then we compute:
\begin{equation} \label{TVOS}
f \mapsto \mathcal{W}_{\psi}f \mapsto \Phi_t(\mathcal{W}_{\psi}f)(\cdot,\cdot) \mapsto
f_{t}^{a}(\cdot):= \int_{S^{1}} \Phi_t^{a}(\mathcal{W}_{\psi}f)(\cdot,\mathbf{n}) \,{\rm d}\mu_{S^1}(\mathbf{n}).
\end{equation}
for $t\geq 0$. Here $U \mapsto W(\cdot,t)=\Phi_t(U)$ denotes the flow operator on $\mathbb{M}$ (\ref{PDE}), but then the PDE in (\ref{PDE}) is re-expressed in the locally adaptive frame (LAD) $\{\mathcal{B}_{i}\}_{i=1}^3$ obtained by the method in  \cite{Fran2009,DuitsJanssen,LieAnalysis}. The PDE then becomes  \\
$\frac{\partial \tilde{W}}{\partial t} = (\|\nabla_{0}\tilde{W}\|^2_0+\varepsilon^2)^{\frac{a}{2}} \sum \limits_{i=1}^3 \tilde{D}_{ii}\;
\mathcal{B}_{i} \circ (\|\nabla_{0} \tilde{W} \|^2_0+\varepsilon^2)^{-\frac{1}{2}}\mathcal{B}_i \tilde{W}, \qquad a\in \{0,1\}$, \\
with $\tilde{D}_{11}=1$, $\tilde{D}_{22}=\tilde{D}_{33}$ as in CED-OS \cite[eq.72]{Fran2009}.
By the invertibility of the orientation score one has $f=f_0^a$
so all flows depart from the original image.

For $a=0$ we call $f \mapsto f_{t}^a$ given by (\ref{TVOS}) a `TVF-OS flow', for $a=1$ we call it a `MCF-OS flow'. In Fig.\!~\ref{fig:quant} we show how
errors progress with
$t\geq 0$. We see that inclusion of LAD is beneficial on the real image.
In Fig.~\ref{fig:qual} we give a qualitative comparison to CED-OS \cite{Fran2009}.
Lines and plateaus are
best preserved by TVF-OS.
\begin{figure}
\centering
\includegraphics[height=0.19\textwidth]{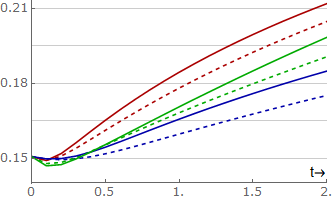}
\quad
\includegraphics[height=0.19\textwidth]{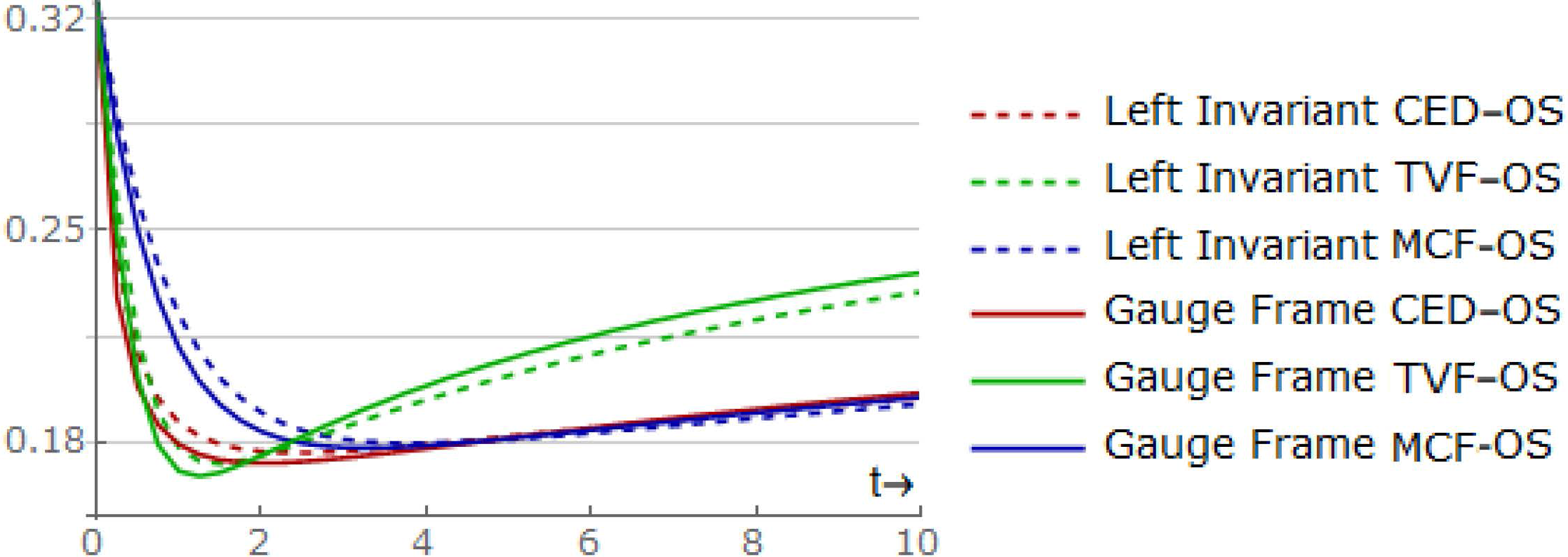}
\vspace{-0.3cm}
\caption{Relative $\mathbb{L}_1$ errors of the spirals test image (left) and the collagen image (right) for the CED-OS, MCF-OS and TVF-OS methods. \label{fig:quant}}
\end{figure}
\vspace{-0.3cm}
\begin{figure}
\centering
\subfloat{\includegraphics[width=0.24\textwidth]{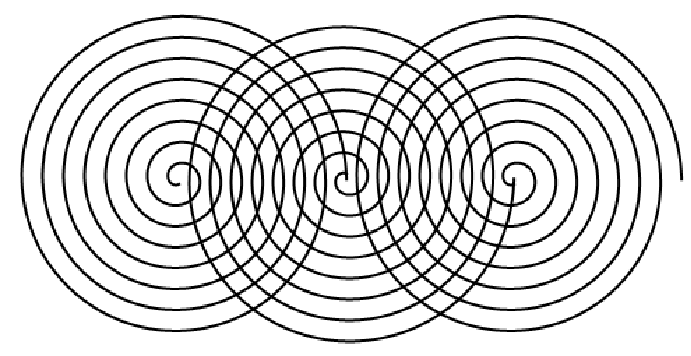}}
\
\subfloat{\includegraphics[width=0.24\textwidth]{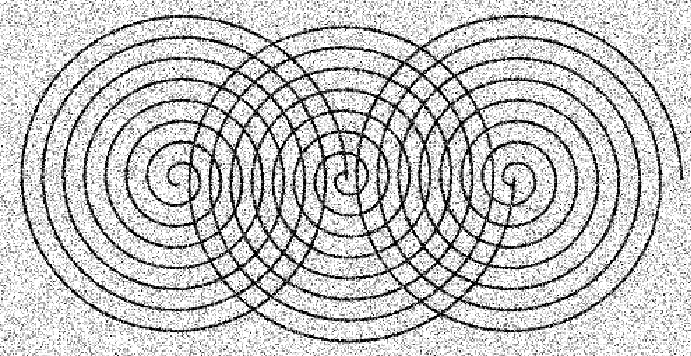}}
\
\subfloat{\includegraphics[width=0.24\textwidth]{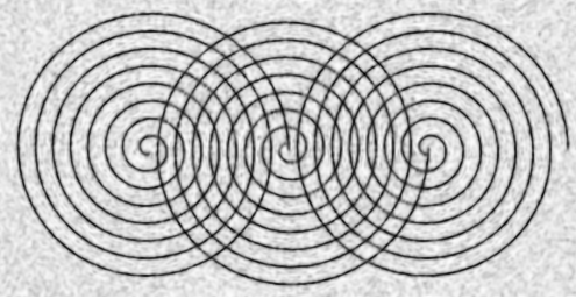}}
\
\subfloat{\includegraphics[width=0.24\textwidth]{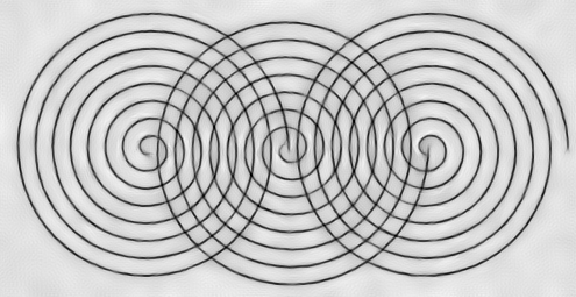}}
\\[0.1em]
\subfloat{\includegraphics[width=0.24\textwidth]{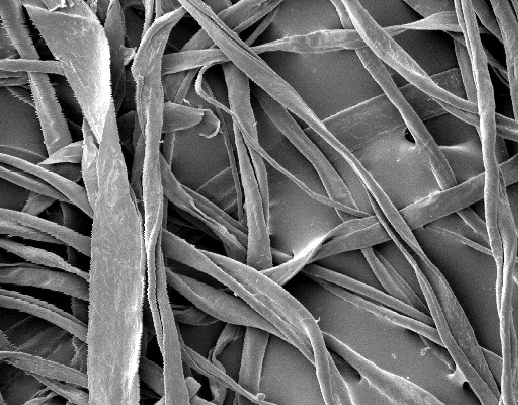}}
\
\subfloat{\includegraphics[width=0.24\textwidth]{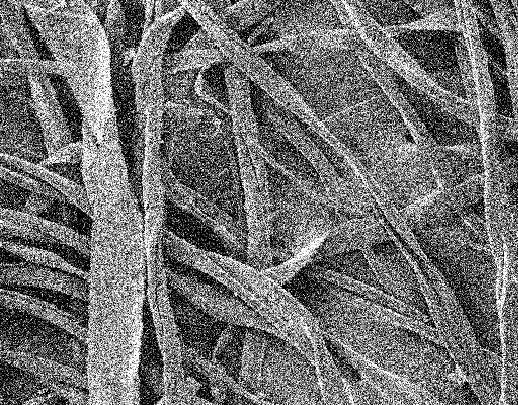}}
\
\subfloat{\includegraphics[width=0.24\textwidth]{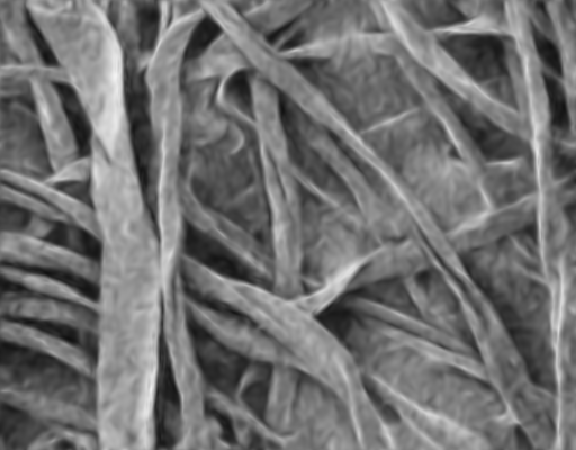}}
\
\subfloat{\includegraphics[width=0.24\textwidth]{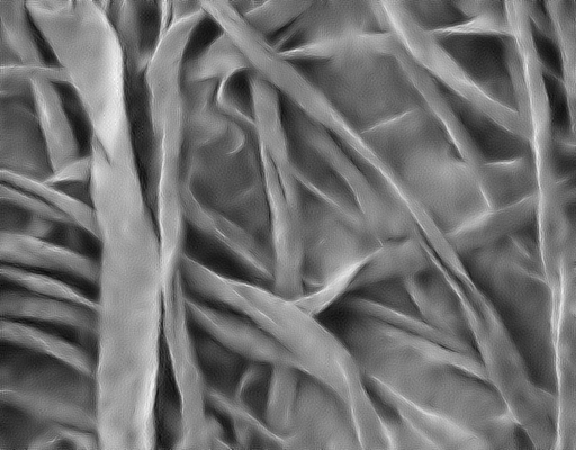}}
\vspace{-0.2cm}
\caption{From left to right; original image; noisy input image $f$,
CED-OS output image \cite{Fran2009}, TVF-OS output image $f_t^{0}$.
We have set $D_A=0.01$, $D_{S}=1$, and took $t=10 \cdot t^*$ for the spirals image and $t=2 \cdot t^*$ for the collagen image where $t^*$ minimizes the relative $\mathbb{L}_{1}$-error to stress different qualitative behavior. \label{fig:qual}}
\end{figure}
\section{Conclusion}
We have proposed a PDE system on the homogeneous space $\mathbb{M}=\mathbb{R}^{d} \rtimes S^{d-1}$ of positions and orientations, for
crossing-preserving denoising and enhancement of (lifted) images containing both complex elongated structures and plateaus.

It includes TVF, MCF and diffusion flows as special cases, and includes (sub-)Riemannian geometry. Thereby we generalized recent related works by Citti et al. \cite{Sanguinetti} and Chambolle \& Pock \cite{Chambolle} from 2D to 3D using a different numerical scheme with new convergence results (Theorem~\ref{th:conv}) and stability bounds. We used
the divergence and intrinsic gradient on a (sub)-Riemannian manifold above $\mathbb{M}$ for a formal weak-formulation of total variation flows, which simplifies if the lifted images are differentiable (Lemma~\ref{lemma:1}).

Compared to previous nonlinear crossing-preserving diffusion methods on $\mathbb{M}$, we showed improvements (Fig.~\ref{fig:quant},\ref{fig:qual}) over CED-OS methods \cite{Fran2009} (for $d=2$) and improvements over contextual fiber enhancement methods in DW-MRI processing (for $d=3$) \cite{CreusenDuits,DuitsJMIV2013} on real medical image data. We observe that crossings and boundaries (of bundles and plateaus) are better preserved over time. We support this quantitatively by a denoising experiment on a benchmark DW-MRI dataset, where MCF performs better than TVF and both perform better than Perona-Malik diffusions, in view of error reduction and stability.
\vspace{-1cm}\mbox{}
\bibliographystyle{splncs04}
\bibliography{literaturebf}
%

\clearpage

\section*{Appendix A: Proof of Theorem~2.}

A functional $\Phi: H \to (-\infty, \infty]$ is said to be $\lambda$-convex for some $\lambda \in \mathbb{R}$ if
\[
u \mapsto \Phi(u) - \frac{\lambda}{2} \| u \|^2
\]
is convex. In that case, the functional
\[
u \mapsto \Phi(u) - \frac{\lambda}{2} \| u - v\|^2
\]
is convex as well, for arbitrary $v \in H$, because the latter functional deviates from the first by an affine functional.

We first prove a stability estimate for the minimization of $1/\tau$-convex functionals.

\begin{lemma}
	\label{le:stability-1-tau-convex}
Let $\tau > 0$. If a functional $\Phi: H \to (-\infty, \infty]$ on $H$ is $1/\tau$-convex, and $u^*$ is its unique minimizer, then for all $u \in H$,
\[
\frac{1}{2\tau} \| u - u^* \|^2 \leq \Phi(u) - \Phi(u^*).
\]	
\end{lemma}
\begin{proof}
The functional $\Psi:H \to (-\infty, \infty]$ given by
\[
\Psi(u) := \Phi(u) - \Phi(u^*) - \frac{1}{2\tau}\| u  - u^*\|^2
\]	
is convex. It is sufficient to show that $\Psi$ is nonnegative. If it were not, there would exist a $v \in H$ such that $\Psi(v) < 0$. We will show that then, for $t$ small enough, $\Phi(t v + (1-t) u^*) < \Phi(u^*)$, contradicting that $u^*$ is a minimizer. We first have by definition that, for $t \in (0,1)$,
\[
\begin{split}
\Phi(t v + (1-t) u^*) &- \Phi(u^*) - \frac{t^2}{2\tau}\| v - u^* \|^2 \\
&=
\Psi(t v + (1-t) u^*).
\end{split}
\]
By the convexity of $\Psi$,
\[
\begin{split}
\Psi(t v + (1-t) u^*)
&\leq t \Psi(v) + (1-t) \Psi(u^*) \\
&= t \Psi(v).
\end{split}
\]
Combining the two inequalities, we find
\[
\Phi(t v + (1-t) u^*) - \Phi(u^*) \leq t \Psi(v) + O(t^2),
\]
so that indeed, for $t $ small enough, $\Phi(tv + (1-t)u^*) < \Phi(u^*)$, leading to the announced contradiction.

Therefore, $\Psi$ is nonnegative, which means that
\[
\frac{1}{2\tau} \|u-u^*\|^2 \leq \Phi(u) - \Phi(u^*)
\]
for all $u \in H$.
\end{proof}
For a proper (i.e. not everywhere equal to $\infty$), lower semicontinuous, convex functional $F$, and $\tau > 0$, define the operator $J_\tau^F: H \to H$ by
\[
J_\tau^F[u_0] := \mathrm{argmin}_{u \in H} \left(\frac{1}{2\tau} \| u - u_0\|^2 + F(u) \right).
\]

\begin{proposition}
\label{pr:bound-gradient-flows}
Let $F, G: H \to [0,\infty]$ be two non-negative, proper, lower semicontinuous, convex functionals on a Hilbert space $H$, such that for all $u \in H$,
\begin{equation}
\label{eq:two-sided-bound}
F(u) - \delta \leq G(u) \leq F(u) + \delta.
\end{equation}
Let $u_0, v_0 \in H$, such that
\begin{equation}
\label{eq:slope-estimate}
|\partial F|(u_0) \leq L \quad \text{ and } \quad  |\partial G|(v_0) \leq L.
\end{equation}
Then, we have the following estimate for the gradient flow $u:[0,\infty) \to H$ of $F$ starting at $u_0$ and the gradient flow $v:[0,\infty) \to H$ of $G$ starting at $v_0$:
\[
\|u(t) - v(t)\| \leq
\begin{cases}
4 \sqrt{\delta t} + \|u_0 - v_0\| & \text{for } 0 \leq t \leq \frac{\delta}{L^2} \\
8 \sqrt[3]{L \delta t^{2}} + \|u_0 - v_0\| & \text{for } t > \frac{\delta}{L^2}.
\end{cases}
\]
\end{proposition}
The idea is that the stability estimate in Lemma \ref{le:stability-1-tau-convex} will allow us to conclude that $J_\tau^F[u_0]$ and $J_\tau^G[v_0]$ are close when $u_0$ and $v_0$ are close. By iterating the operators $J_\tau^F$ and $J_\tau^G$, we approximate the gradient flows of $F$ and $G$ respectively, and from the slope estimate (\ref{eq:two-sided-bound}) we will derive that this approximation is uniform. This will allow us to derive bounds for the gradient flows from the bounds for $J_\tau^F$ and $J_\tau^G$.
\begin{proof}	
Let $\tau>0$ and let $u_1^F := J^F_\tau [u_0]$ and $v_1^G := J^G_\tau[v_0]$. Set also $v_1^F:=J^F_\tau[v_0]$ and $u_1^G := J^G_\tau[u_0]$. Then, using the definition of $v_1^F$ in the second inequality below, we find
\[
\begin{split}
\frac{1}{2\tau} &\| v_1^F - v_0 \|^2 + G(v_1^F)\\
& \overset{(\ref{eq:two-sided-bound})}{\leq} \frac{1}{2\tau} \| v_1^F - v_0 \|^2 + F(v_1^F)+ \delta\\
& \leq \frac{1}{2\tau}\|v_1^G - v_0\|^2 + F(v_1^G) + \delta \\
& \overset{(\ref{eq:two-sided-bound})}{\leq} \frac{1}{2\tau}\|v_1^G - v_0\|^2 + G(v_1^G) + 2 \delta.
\end{split}
\]
Because the functional
\[
v \mapsto \frac{1}{2 \tau} \|v -v_0\|^2 + G(v)
\]
is $1/\tau$-convex, it follows by Lemma \ref{le:stability-1-tau-convex} that
\[
\frac{1}{2\tau} \| v_1^F - v_1^G \|^2 \leq
2 \delta.
\]

Now we use that $J_\tau^F$ is non-expansive \cite[Eq.~(4.0.2)]{Ambrosio}, so
\[
\| u_1^F - v_1^F \| = \| J_\tau^F(u_0) - J_\tau^F(v_0) \| \leq \| u_0 - v_0\|.
\]
We conclude that
\[
\|u_1^F - v_1^G \| \leq \|  u_0 - v_0  \| + 2 \sqrt{\delta \tau}.
\]
By iterating this estimate, we derive
\begin{equation}
\label{eq:close-discrete}
\| (J_\tau^F)^n[u_0] - (J_\tau^G)^n[v_0] \|
\leq \|u_0 - v_0 \| + 2 n \sqrt{ \delta \tau }.
\end{equation}
The a priori estimate \cite[Theorem 4.0.4, (v)]{Ambrosio} yields that the gradient flows $u$ and $v$ of $F$ and $G$ respectively are approximated well by $(J_{t/n}^F)^n[u_0]$ and $(J_{t/n}^G)^n[v_0]$. More precisely, for $t > 0$ and $n > 0$, the a priori estimate gives
\[
\left\| u(t) - (J_\tau^F)^n\right\| \leq \frac{L t}{\sqrt{2}n} \quad \text{ and } \quad \left\| v(t) - (J_\tau^G)^n\right\| \leq \frac{L t}{\sqrt{2}n}.
\]
By these a priori estimates and the estimate for the discrete flows (\ref{eq:close-discrete}), we see that
\[
\begin{split}
\| u(t) - v(t) \|
&\leq \| u(t) - (J_{t/n}^F)^n[u_0] \| + \| v(t) - (J_{t/n}^F)^n[v_0] \|
\\
&\qquad + \| (J_{t/n}^F)^n[u_0] - (J_{t/n}^G)^n[v_0] \|
 \\
&\leq \sqrt{2}L \frac{t}{n} + 2n \sqrt{\frac{\delta t}{n}} + \|u_0 - v_0\|.
\end{split}
\]
To derive the final estimates, we need to make good choices for $n$. If $0 \leq t \leq \delta / L^2$, we take $n=1$ and obtain
\[
\begin{split}
\|u(t) - v(t)\| &\leq \sqrt{2} L t + 2 \sqrt{ \delta t} + \|u_0 - v_0\| \\
&\leq 4 \sqrt{\delta t} + \|u_0 - v_0\|.
\end{split}
\]
If $t > \delta/L^2$, we choose $n =\lceil L^{2/3} (t/\delta)^{1/3} \rceil$, which is larger than or equal to $2$. In that case,
\[
n/2 \leq n-1 < L^{2/3} (t/\delta)^{1/3} \leq n.
\]
We then obtain
\[
\|u(t) - v(t) \| \leq 8 L^{1/3} \delta^{1/3} t^{2/3} + \|u_0 - v_0\|.
\]
\end{proof}

We now know that the gradient flows of $F$ and $G$ are close when the slopes $|\partial F|(u_0)$ and $|\partial G|(v_0)$ are bounded. This assumption can be rather stringent. We will relax it, and merely require that $F(u_0)$ and $G(v_0)$ are bounded by some constant $E > 0$, in exchange for a bound between gradient flows that is slightly worse. Our approach will be to run the gradient flow for a small time $s$ from $u_0$ and $v_0$, and use the regularizing property of the gradient flow to conclude a slope bound. On the other hand, if $s$ is small, $u(s)$ and $v(s)$ will be close to $u_0$ and $v_0$. We will then choose $s$ (almost) optimally to derive a bound between the gradient flows.
%
\begin{theorem}
	Let $F:H \to [0,\infty]$ and $G:H \to [0,\infty]$ be two proper, lower semicontinuous, convex functionals on a Hilbert space $H$, such that
	\[
	F(u) - \delta \leq G(u) \leq F(u) + \delta
	\]
	for all $u \in H$. Let $u_0, v_0 \in H$ be such that $F(u_0) \leq E$ and $G(v_0) \leq E$ and $\|u_0 - u^*\| \leq M$ and $\|v_0 - v^*\| \leq M$, for some constants $E, M > 0$, where $u^*$ and $v^*$ minimize $F$ and $G$ respectively. The gradient flow $u:[0,\infty) \to H$ of $F$ starting at $u_0$, and the gradient flow $v:[0,\infty) \to H$ of $G$ starting at $v_0$ satisfy
	\[
	\| u(t) - v(t) \| \leq 16 (M E \delta t^2)^{1/5} + \| u_0 - v_0 \|
	\]
	for all $0 \leq t < E^6 M^6/\delta^9$.
\end{theorem}

\begin{proof}
By the Evolution Variational Inequality \cite[Theorem 4.0.4, (iii)]{Ambrosio}, we know that for all $s > 0$
\begin{subequations}
	\label{eq:close-to-initial}
\begin{equation}
\|u(s) - u_0\| \leq \sqrt{2 s F(u_0) }
\end{equation}
and
\begin{equation}
\|v(s) - v_0 \| \leq \sqrt{2 s G(v_0) }.
\end{equation}
\end{subequations}

By the regularizing property \cite[Theorem 4.0.4, (ii)]{Ambrosio},
\begin{subequations}
	\label{eq:slope-estimates}
\begin{equation}
|\partial F|(u(s)) \leq \frac{1}{s}\| u_0 - u^* \| \leq \frac{M}{s}
\end{equation}
and
\begin{equation}
|\partial G|(v(s)) \leq \frac{1}{s} \|v_0 - v^* \| \leq \frac{M}{s}
\end{equation}
\end{subequations}
where $u^*$ minimizes $F$ and $v^*$ minimizes $G$.

Because the gradient flow is a non-expansive semigroup \cite[Theorem 4.0.4, (iv)]{Ambrosio}, we get
\[
\begin{split}
\| u(t) - v(t) \|
&\leq \| u(t+s) - v(t+s) \| + \| u(t+s) - u(t) \| + \| v(t+s) - v(t) \|\\
&\leq \| u(t+s) - v(t+s) \| + \| u(s) - u_0 \| + \|v(s) - v_0 \|.
\end{split}
\]

Now assume $t < E^6 M^6/\delta^9$.
We will want to choose $s$ (almost) optimally, depending on $t$. We choose
\[
s = \frac{M^{2/5} \delta^{2/5} t^{4/5}}{E^{3/5}}
\]
and note that with $L := M/s$, we have
\[
t \geq \frac{\delta}{L^2}.
\]
By the slope estimates (\ref{eq:slope-estimates}) we can apply Proposition \ref{pr:bound-gradient-flows} to the gradient flows starting at $u(s)$ and $v(s)$, to get
\[
\begin{split}
\|u(t) - v(t) \|&\leq 8 M^{1/3} s^{-1/3} \delta^{1/3} t^{2/3} + \|u(s) - v(s) \| + \| u(s) - u_0 \|
+ \| v(s) - v_0 \| \\
&\leq 8 M^{1/3} s^{-1/3} \delta^{1/3} t^{2/3} + 2 \| u(s) - u_0 \| + 2\|v(s) - v_0\| + \| u_0 - v_0\|\\
&\overset{(\ref{eq:close-to-initial})}{\leq} 8 M^{1/3} s^{-1/3} \delta^{1/3} t^{2/3} + \sqrt{ 32 s E } + \| u_0 - v_0 \| \\
&= 16 M^{1/5} \delta^{1/5} t^{2/5} E^{1/5} + \|u_0 - v_0\|.
\end{split}
\]
\end{proof}

\end{document}